\documentclass[12pt]{article}
\usepackage{graphicx,epsfig}
\usepackage{amssymb,amsmath,amscd,amsthm}

\usepackage{graphicx}
\usepackage[active]{srcltx}

\newtheorem{theorem}{Theorem}[section]

\newtheorem{lemma}[theorem]{Lemma}
\newtheorem{proposition}[theorem]{Proposition}

\newtheorem{remark}[theorem]{Remark}

\date{}

\begin{document}

\title{On the Critical Behavior of Continuous Homopolymers}
\author{Leonid Koralov, Zsolt Pajor-Gyulai}
\maketitle

\vskip1cm
\begin{abstract}
The aim of this paper is to investigate the distribution of a
continuous homopolymer in the presence of an attractive finitely
supported potential. The most intricate behavior can be observed
if we simultaneously vary two parameters: the temperature, which
approaches the critical value, and the length of the polymer,
which tends to infinity. As the main result, we identify the
distributions that appear in the limit (after a diffusive scaling
of the original polymer measures) and depend on the relation
between the two parameters.

\end{abstract}

\section{Introduction}

We consider the following model of long homogeneous polymer chains
in an attractive potential field. Let the space
$C([0,T],\mathbb{R}^d)$  be equipped with the Wiener measure
$\mathrm{P}_{0,T}$ and suppose that we have a smooth, nonnegative,
not identically equal to zero, compactly supported potential $v\in
C_0^{\infty}(\mathbb{R}^d)$ and a coupling constant $\beta\geq 0$
(inverse temperature), which regulates the strength of the
attraction. The elements $\omega(.)$ of the space are interpreted
as realizations of a continuous homopolymer on $[0,T]$ and are
distributed according to the Gibbs measure $\mathrm{P}_{\beta,T}$
with
\[
\frac{d\mathrm{P}_{\beta,T}}{d\mathrm{P}_{0,T}}(\omega)=
\frac{e^{\beta\int_0^Tv(\omega(t))dt}}{Z_{\beta,T}},\qquad
\omega\in C([0,T],\mathbb{R^d}),
\]
where
\[
Z_{\beta,T}=\mathrm{E}_{0,T}e^{\beta\int_0^Tv(\omega(t))dt}
\]
is the partition function.

Thus the polymer measure we consider is of a ``mean field" type,
where the polymer chain interacts with the external attractive
potential (as in \cite{LGK}). While the potential is assumed to be
constant in time in our paper, many interesting results have been
obtained for disordered media, i.e., time-dependent random
potentials (existence of phase transitions, dependence of the
growth rate of the partition function on the temperature, etc.,
see \cite{AS}, \cite{Gia}, for example). We suspect, however, that
the disordered models are not amenable to such a detailed analysis
of the phase transition phenomena as found in the current paper.

It is established - using the Feynman-Kac formula - that under the
measures $ \mathrm{P}_{\beta,T}$, the processes $\{\omega(t),
0\leq t\leq T\}$ are time-inhomogeneous and Markovian and that
their transition densities can be expressed in terms of the
fundamental solution $p_{\beta}$ of the parabolic equation
\begin{equation} \label{frst}
\frac{\partial u}{\partial t}=H_{\beta}u, ~~~{\rm where}~~~~
H_{\beta}=\frac{1}{2}\Delta  +\beta v: L^2(\mathbb{R}^d)\to
L^2(\mathbb{R}^d).
\end{equation}
More precisely, the finite-dimensional distributions are
\begin{align} \label{fdd}
\mathrm{P}_{\beta,T}(\omega(t_1)\in A_1,...,\omega(t_n)\in A_n)=\\
\nonumber =\frac{1}{Z_{\beta,T}}\int_{A_1}...\int_{A_n}&
\int_{\mathbb{R}^d}p_{\beta}(t_1,0,x_1)...p_{\beta}(T-t_n,x_n,y)
dydx_n...dx_1
\end{align}
for $0\leq t_1\leq...\leq t_n\leq T$ and $A_1,...,A_n\in \mathcal{B}(\mathbb{R}^3)$.

It is worth noting that
\[
Z_{\beta,t}=\int_{\mathbb{R}^3}p_{\beta}(t,0,y)dy.
\]

It was shown that in $d\geq 3$, at a ceratin (critical) value of
the coupling constant, there occurs a phase transition occurs
between a densely packed globular state and an extended phase of
the polymer. Namely, for $\beta > \beta_{\rm cr}$, a typical
polymer realization is at a distance of order one from the origin
as $T\to\infty$. On the other hand, for $\beta \leq \beta_{\rm
cr}$, the realizations need to be scaled by the factor $\sqrt{T}$
in the spatial variables in order to get a non-trivial limit (as
in Theorem \ref{the1a} below). The critical value of the coupling
constant is
\[
\beta_{\rm cr}=\sup\{\beta>0|\sup\sigma(H_{\beta})=0\},
\]
where  $\sigma(H_{\beta})$ is the spectrum of the operator  $ H_{\beta}$.

Here's the precise statement of the result proved in \cite{CKMV09} that
is relevant to this paper.

\begin{theorem} \label{the1a}
For $a>0$, let $f_a:C([0,T],\mathbb{R}^3)\to
C([0,aT],\mathbb{R}^3)$ be the mapping defined by
$(f_a\omega)(t)=\sqrt{a}\omega(t/a)$, and let $f_a\mathrm{P}$ be
the push-forward of a measure $\mathrm{P}$ by this mapping. There
is a measure $ \mathrm{Q}_0$ on $C([0,1],\mathbb{R}^3)$ that
corresponds to a certain time-inhomogeneous Markov process
starting at the origin such that
\[
f_{T^{-1}}{\mathrm{P}}_{\beta_{\rm cr},T}\Rightarrow {\mathrm{Q}}_0~~~{\rm as}~~ T\to\infty.
\]
\end{theorem}

\begin{remark}
We use the subscript $0$ in the notation for the limiting measure
since we will introduce a whole class of measures, and
${\mathrm{Q}}_0$ will be just a particular member of this class.
\end{remark}

\begin{remark}
For $\beta < \beta_{\rm cr}$, the measures
$f_{T^{-1}}{\mathrm{P}}_{\beta,T}$ converge to the Wiener measure
on $C([0,1],\mathbb{R}^3)$. Note, however, that this convergence
and the convergence in Theorem~\ref{the1a} take place for
fixed values of $\beta$ that don't scale with $T$.
\end{remark}

In this paper we'll allow $\beta = \beta(T)$ to change in such a
way that $(\beta(T)-\beta_{\rm cr})\sqrt{T}\to\chi$. We'll show
that there are measures $ \mathrm{Q}_{\gamma}$  and a linear
mapping $\gamma(\chi) = c \chi$ such that
\[
f_{T^{-1}}{\mathrm{P}}_{\beta(T),T}\Rightarrow{\mathrm{Q}}_{\gamma(\chi)}~~~{\rm as}~~{T \rightarrow \infty}.
\]
Before formulating this as a theorem, let us describe the limiting
measures~$ \mathrm{Q}_{\gamma}$. They were introduced in
\cite{CKMV10} as the polymer measures  on $C([0,1],\mathbb{R}^3)$
corresponding to zero-range attracting potentials (i.e., the
potentials that are, roughly speaking, concentrated at the
origin). More precisely, it was shown in \cite{CKMV10} that,
 for each $\gamma \in \mathbb{R}$, the
polymer measures corresponding to the potentials
\begin{equation}
\label{poten} v_\gamma^\varepsilon(x) = (\frac{\pi^2}{8\varepsilon^2}
+ \frac{\gamma}{\varepsilon}) \chi_{B} (\frac{x}{\varepsilon})
\end{equation}
converge as $\varepsilon \downarrow 0$. (Here $\chi_{B}$ is the
indicator function of the unit ball centered at the origin.) The
limit will be denoted by $\mathrm{Q}_\gamma$. It is worth noting
that all the measures $Q_{\gamma}$ are spherically symmetric and
thus the dependence of the polymer measures
$\mathrm{P}_{\beta(T),T}$ on the `shape' of the potential
disappears in the limit.

Another, less intuitive, but more convenient, way to define the
measures $ \mathrm{Q}_{\gamma}$ is through the self-adjoint
extensions of the Laplacian acting on $ \mathbb{R}^3 \setminus 0$.
Namely, it was shown in \cite{AGHH05} that there is a
one-parameter family $\{\mathcal{L}_{\gamma},
\gamma\in\mathbb{R}\}$ of self-adjoint operators acting on $ L^2(
\mathbb{R}^3)$ such that $ \mathcal{L}_\gamma f  = \Delta f$
whenever $f \in C_0^\infty (\mathbb{R}^3 \setminus 0)$. The kernel
of $\exp( t \mathcal{L}_{\gamma})$, $t > 0$, is given by
\begin{equation}\label{fund_zerorange}
\overline{p}_{\gamma}(t,x,y)=\frac{e^{-|x-y|^2/2t}}{(2\pi
t)^{3/2}}+\frac{1}{4\pi^2
i}\int_{\Gamma(a)}\frac{e^{-\sqrt{2\lambda}(|x|+|y|)+\lambda
t}}{(\sqrt{2\lambda}-\gamma)|x||y|}d\lambda,
\end{equation}
where $x,y \neq 0$, $\Gamma(a)=\{z\in\mathbb{C}|{\rm Re} z=a\}$,
and $a>\gamma^2/2$. Thus $\overline{p}_{\gamma}(t,x,y)$ can be
interpreted as the formal fundamental solution  of the parabolic
equation
\[
\frac{\partial u}{\partial t}=\mathcal{L}_{\gamma}u.
\]
By analogy with (\ref{fdd}), we can define the measures $
\overline{\mathrm{P}}^x_{\gamma,T}$, whose finite-dimensional
distributions are given by the formula
\begin{align*}
\overline{\mathrm{P}}^x_{\gamma,T}&(\omega(t_1)\in A_1,...,\omega(t_n)\in A_n)=\\
&=\frac{1}{\overline{Z}_{\gamma,T}(x)}\int_{A_1}...\int_{A_n}\int_{\mathbb{R}^d}\overline{p}_{\gamma}
(t_1,x,x_1)...\overline{p}_{\gamma}(T-t_n,x_n,y)dydx_n...dx_1
\end{align*}
for $0\leq t_1\leq...\leq t_n\leq T$ and $A_1,...,A_n\in\mathcal{B}(\mathbb{R}^3)$, where
\[
{\overline{Z}}_{\gamma,T}(x)=\int_{\mathbb{R}^3}\bar{p}_{\gamma}(t,x,y)dy
= 1+\frac{1}{2\pi i}\int_{\Gamma(a)}\frac{e^{\lambda
t}}{\sqrt{2\lambda}-\gamma}\frac{e^{-\sqrt{2\lambda}|x|}}{\lambda|x|}d\lambda.
\]
Note the dependence of the measures on the initial point $x \in
\mathbb{R}^3$. In fact, neither $\overline{p}_{\gamma}(t_1,x,y)$
nor $ {\overline{Z}}_{\gamma,T}(x)$ are defined for $x = 0$, but
we can make sense out of $\overline{\mathrm{P}}^0_{\gamma,T}$ by
taking the limit of $\overline{\mathrm{P}}^x_{\gamma,T}$ as $x
\rightarrow 0$. If we put $ \mathrm{Q}_\gamma =
\overline{\mathrm{P}}^0_{\gamma,1}$, we'll obtain the same family
of measures corresponding to zero-range potentials.

It is also worth mentioning that the assumption $d = 3$ is
important. Indeed, there is no phase transition for $d=1,2$, while
for dimensions higher than 3 there are no nontrivial self-adjoint
extensions of the Laplacian.

\section{The main result}
It was shown in Lemmas 5.4 and 5.5 of \cite{CKMV09} that the solution space of the problem
\[
\frac{1}{2}\Delta\psi+\beta_{\rm cr} v(x)\psi=0,\qquad\psi(x)=\mathcal{O}(|x|^{-1}),
\qquad\frac{\partial\psi}{\partial r}(x)=\mathcal{O}(|x|^{-2})~~~{\rm as}~~|x|\to\infty
\]
is one-dimensional, and $\psi$ can be chosen to be positive.

The main result of our paper is the following.
\begin{theorem}\label{main_result}
If $(\beta(T)-\beta_{\rm cr})\sqrt{T}\to\chi$, then
\[
f_{T^{-1}}\mathrm{P}_{\beta(T),T}\Rightarrow \mathrm{Q}_{\gamma(\chi)}~~~{\rm as}~~T\to\infty,
\]
where $\gamma(\chi) = c \chi$ with
\[
c = {\sqrt{2}}/({\beta_{\rm cr}^2\gamma_1})~~~{\rm and}~~~
\gamma_1=\frac{(\int_{\mathbb{R}^3}v(x)\psi(x)dx)^2}{\sqrt{2}\pi\int_{\mathbb{R}^3}v(x)\psi^2(x)dx}.
\]
\end{theorem}
The proof will rely on the asymptotic formulas for the fundamental solution of (\ref{frst}) and the partition
function, which can be found in the right-hand side of (\ref{fdd}). We formulate these asymptotic formulas
in several propositions below.

Let us introduce the space
\[
C_{\rm exp}(\mathbb{R}^3)=\left\{f\in
C(\mathbb{R})\bigg|||f||_{C_{\rm exp}}:=
\sup_{x\in\mathbb{R}^3}(|f(x)|e^{|x|^2})<\infty\right\}
\]
For $f\in C_{\rm exp}(\mathbb{R}^3)$, the solution of the parabolic problem

\begin{equation}\label{elliptic_eq}
\frac{\partial u}{\partial t}=H_{\beta(T)}u,\qquad u(0,x)=f(x)
\end{equation}
is given by the inverse Laplace transform
\begin{equation} \label{ltra}
u_{\beta(T)}(t,x)=-\frac{1}{2\pi i}
\int_{\Gamma(\lambda_0(\beta(T))+\delta/T)}e^{\lambda t}(R_{\beta(T)}(\lambda)f)(x)d\lambda,
\end{equation}
where $\lambda_0(\beta)=\sup\sigma(H_{\beta})$, $\delta>0$, and
$R_{\beta}(\lambda)=(H_{\beta}-\lambda I)^{-1}$ is the resolvent
of $H_{\beta}$.

\begin{proposition}\label{cauchy1}
For $\epsilon>0$,  we have
\begin{enumerate}
\item (The solution of the Cauchy problem)
\begin{equation}\label{thm2result}
u_{\beta(T)}(t,x)=\frac{1}{\sqrt{T}} \frac{\alpha(f)}{2\pi i}
\int_{\Gamma(\frac{\gamma^2}{2}+\delta)}
\frac{e^{\frac{\lambda t}{T}}}{\sqrt{2\lambda}-\gamma}
\frac{e^{-\sqrt{2\lambda}\frac{|x|}{\sqrt{T}}}}{|x|}d\lambda+q^f(T,t,x)
\end{equation}
with $\gamma={\sqrt{2}\chi}/({\gamma_1\beta_{\rm cr}^2})$, $\delta>0$,
and
\[
\alpha(f)=\kappa\int_{\mathbb{R}^3}\psi(x)f(x)dx,\qquad\kappa=
\frac{1}{\beta_{\rm cr}\int_{\mathbb{R}^3}v(x)\psi(x)dx},
\]
where the error term satisfies
\begin{equation}\label{thm2error}
\sup_{\epsilon T\leq t\leq
T,\epsilon\sqrt{T}\leq|x|\leq\epsilon^{-1}\sqrt{T}}|q^f(T,t,x)|=
||f||_{C_{\rm exp}}\mathcal{O}(T^{-3/2}),
\end{equation}

\item (The fundamental solution)
\begin{equation}\label{thm2fundam}
p_{\beta(T)}(t,y,x)=\frac{1}{\sqrt{T}}\frac{\kappa\psi(y)}{2\pi i}
\int_{\Gamma\left(\frac{\gamma^2}{2}+\delta\right)}
\frac{e^{\frac{\lambda{t}}{T}}}{\sqrt{2\lambda}-\gamma}
\frac{e^{-\sqrt{2\lambda}\frac{|x|}{\sqrt{T}}}}{|x|}d\lambda+q(T,t,y,x),
\end{equation}
where
\[
\sup_{\epsilon T\leq t\leq T,\epsilon\sqrt{T}\leq|x|\leq\epsilon^{-1}
\sqrt{T},|y|\leq \epsilon^{-1}}T|q(T,t,y,x)|\to 0.
\]
\end{enumerate}
\end{proposition}

We'll also need the asymptotics of the fundamental solution when $y$ is of order~$\sqrt{T}$.

\begin{proposition}\label{cauchy2}
For $\epsilon>0$,
\begin{equation}\label{thm2asym}
p_{\beta(T)}(t,y,x)=p_0(t,y,x)+\frac{1}{\sqrt{T}}\frac{1}{4\pi^2i}
\int_{\Gamma(\frac{\gamma^2}{2}+\delta)}\frac{e^{\frac{\lambda t}{T}-\sqrt{2\lambda}
\frac{|x|+|y|}{\sqrt{T}}}}{(\sqrt{2\lambda}-\gamma)|x||y|}d\lambda+q(T,t,y,x),
\end{equation}
where $p_0$ is the fundamental solution of the heat equation and
\begin{equation}\label{err_term}
\sup_{\epsilon T\leq t\leq T,\epsilon\sqrt{T}\leq |y|,|x|\leq\epsilon^{-1}\sqrt{T}}T^{3/2}q(T,t,y,x)\to 0.
\end{equation}
\end{proposition}

We will need one more result concerning the behavior of the
partition sum with respect to the measure $\mathrm{P}_{0,T}^x$
(the Wiener measure $\mathrm{P}_{0,T}$ translated by the vector
$x$).

\begin{proposition}\label{partition_sum}
If $Z_{\beta,t}(x)=\mathrm{E}_{0,T}^x e^{\beta
\int_0^Tv(\omega(t))d t}$, then for every $\epsilon>0$, we have
\[
Z_{\beta(T),t}(x)=1+\frac{\sqrt{T}}{2\pi
i}\int_{\Gamma\left(\frac{\gamma^2}{2}+\delta\right)}
\frac{e^{\frac{\lambda t}{T}}}{\sqrt{2\lambda}-\gamma}
\frac{e^{-\sqrt{2\lambda}\frac{|x|}{\sqrt{T}}}}{\lambda|x|}d\lambda+q^Z(T,t,x),
\]
where
\[
\lim_{T\to\infty}\sup_{\epsilon\sqrt{T}\leq|x|\leq\epsilon^{-1}\sqrt{T},\epsilon
T\leq t\leq T}|q^Z(T,t,x)|=0.
\]
\end{proposition}

\section{Proofs}

Note that the expression (\ref{ltra}) for the solution of
(\ref{elliptic_eq}) contains the resolvent of the operator
$H_\beta$ inside the integral. It will be seen that the main
contribution to the integral comes from the values of $\lambda$
that are close to zero. Therefore, it is important to know the
asymptotics of the resolvent as $\lambda \rightarrow 0$. First,
let us make several observations concerning the spectrum of
$H_\beta$ (a more detailed discussion of the spectral properties
of $H_\beta$ together with the proofs of the following four lemmas can be found in \cite{CKMV09}).

It is well known that for some $0\leq N<\infty$
\[
\sigma(H_{\beta})=(-\infty,0]\cup\{\lambda_j\}_{j=0}^N,
\]
where the eigenvalues are enumerated in decreasing order.

\begin{lemma}
For $\beta\leq\beta_{\rm cr}$, we have $\sup\sigma(H_{\beta})=0$,
while $\sup\sigma(H_{\beta})=\lambda_0(\beta)>0$ for
$\beta>\beta_{\rm cr}$. In the latter case, $\lambda_0(\beta)$ is
a simple eigenvalue. It is a strictly increasing and continuous
function of $\beta$. Moreover, $\lim_{\beta\downarrow\beta_{\rm
cr}}\lambda_0(\beta)=0$ and
$\lim_{\beta\uparrow\infty}\lambda(\beta)=\infty$.
\end{lemma}

Due to the monotonicity and continuity of
$\lambda=\lambda_0(\beta)$ for $\beta>\beta_{\rm cr}$, we can
define the inverse function
\[
\beta=\beta_0(\lambda):[0,\infty)\to[\beta_{\rm cr},\infty).
\]
Let $\mathbb{C}'=\mathbb{C}\backslash (-\infty, 0]$.

The resolvent $R_{\beta}(\lambda)=(H_{\beta}-\lambda I)^{-1}$ is a
meromorphic operator valued function on $\mathbb{C}'$. Let us
introduce the operator
\[
A(\lambda)=v(x)R_0(\lambda):\qquad C_{\rm exp}(\mathbb{R}^d)\to
C_{\rm exp}(\mathbb{R}^d),\qquad \lambda\in\mathbb{C}'.
\]
We have the following identity for the resolvent
\begin{equation}
R_{\beta}(\lambda)=R_0(\lambda)(I+\beta A(\lambda))^{-1}, \lambda\in\mathbb{C}'.
\end{equation}
From this it is not difficult to show that $1/\beta_0(\lambda)$ is
the principal eigenvalue of $-A(\lambda)$ and using this, we can
extend the domain of $\beta_0$ to
$[0,\infty)\cup(U\cap\mathbb{C}')$, where $U$ is a sufficiently
small neighborhood of zero.
\begin{lemma}\label{princip_eig_inv_asym}
(Asymptotic behavior)
\begin{enumerate}
\item[a)] The principal eigenvalue has the following behavior as $\beta\downarrow\beta_{\rm cr}$
\[
\lambda_0(\beta)=\frac{1}{\gamma^2\beta_{\rm
cr}^4}(\beta-\beta_{\rm cr})^2(1+{o}(1)).
\]
\item[b)] The small $\lambda$ asymptotics of $\beta_0(\lambda)$ is given by
\begin{equation}
\frac{1}{\beta_0(\lambda)}=\frac{1}{\beta_{\rm
cr}}-\gamma_1\sqrt{\lambda}+\mathcal{O}(\lambda),\qquad\lambda\to
0, \lambda\in\mathbb{C}'.
\end{equation}
\end{enumerate}
\end{lemma}

\begin{lemma}\label{op_bdd}
We have
\begin{itemize}
\item[a)] For every $\epsilon>0$, $R_0(\lambda):C_{\rm exp}(\mathbb{R}^3)\to C(\mathbb{R}^3)$
is a bounded operator with
$||R_0(\lambda)||=\mathcal{O}(|\lambda|^{-1})$ as
$\lambda\to\infty$, $|\arg \lambda|\leq\pi-\epsilon$.

\item[b)] The operator $(I+\beta A(\lambda))^{-1}$ on $C_{\rm exp}(\mathbb{R}^3)$ is meromorphic
in $\mathbb{C}'$. For each $\epsilon,\Lambda>0$, there is a $\tilde{\delta}>0$ such
that it is uniformly bounded operator in $\lambda\in\mathbb{C}'$,
$|\arg\lambda|\leq\pi-\epsilon$,$|\lambda|\geq\Lambda$,
$|\beta-\beta_{\rm cr}|\leq\tilde{\delta}$.
\end{itemize}
\end{lemma}

As for the $\lambda\to 0$ asymptotics,

\begin{lemma}\label{opasym}
There are $\lambda_0>0$ and $\delta_0>0$ such that for
$\lambda\in\mathbb{C}'\cup\{0\}, |\lambda|\leq\lambda_0,
|\beta-\beta_{\rm cr}|\leq\delta_0, \beta\neq\beta_0(\lambda)$, we
have
\[
(I+\beta
A(\lambda))^{-1}=\frac{\beta_0(\lambda)}{\beta_0(\lambda)-\beta}
(B+S(\lambda))+C(\lambda,\beta)
\]
as operators on $C_{\rm exp}(\mathbb{R}^3)$, where $B$ is the one
dimensional operator with kernel
\[
B(x,y)=\frac{v(x)\psi(x)\psi(y)}{\int_{\mathbb{R}^3}v(x)\psi^2(x)dx},
\]
$S=\mathcal{O}(\sqrt{|\lambda|})$, $S(0)=0$, and
$C(\lambda,\beta)$ is bounded uniformly in $\lambda$ and $\beta$.
\end{lemma}

Now for $f\in C_{\rm exp}(\mathbb{R}^3)$, define
\[
g_T^f(z,y)=(I+\beta(T)A(z))^{-1}f(y)
\]
The key to the proofs in this paper is the following lemma.

\begin{lemma}
For $\delta>0$,
\begin{equation}\label{lemma1statement}
g_T^{f}\left(\frac{\lambda}{T},y\right)=\sqrt{T}\frac{2\pi
\int_{\mathbb{R}^3}\psi(x)f(x)dx}{\beta_{\rm cr}\left(\int_{\mathbb{R}^3}v(x)\psi(x)dx\right)^2}\frac{v(y)\psi(y)}{\sqrt{2\lambda}-\gamma}+(K(\lambda,T)f)(y)
\end{equation}
as $T\to\infty$ for $\lambda\in\Gamma(\gamma^2/2+\delta)$, where $K(\lambda,T)$ is uniformly bounded as an operator on $C_{\rm exp}(\mathbb{R}^3)$.
\end{lemma}

\begin{proof}
First consider the $|\lambda|<aT$ case for $a$ sufficiently small
in order for the following to make sense. Lemma
\ref{princip_eig_inv_asym} yields
\begin{equation}
\beta_0(z)=\beta_{\rm cr}+\beta_{\rm cr}^2\gamma_1\sqrt{z}+\mathcal{O}(z),\qquad z\to 0, z\in\mathbb{C}'.
\end{equation}
Using that
\begin{equation}\label{betaTasym}
\beta(T)=\beta_{\rm cr}+\frac{\chi}{\sqrt{T}}+o\left(\frac{1}{\sqrt{T}}\right),\qquad T\to\infty,
\end{equation}
we can write for $\lambda\in\Gamma(\gamma^2/2+\delta)$ that
\begin{equation}
\beta_0\left(\frac{\lambda}{T}\right)-\beta(T)=\beta_{\rm
cr}^2\gamma_1\sqrt{\frac{\lambda}{T}}-\frac{\chi}{\sqrt{T}}+
\mathcal{O}\left(\frac{\lambda}{T}\right)+o(1/\sqrt{T})
\end{equation}
as $T\to\infty$. Using the definition of $\gamma$, this can be rewritten as
\begin{align}\label{asymps}
\frac{\beta_0\left(\frac{\lambda}{T}\right)}{\beta_0\left(\frac{\lambda}{T}\right)-\beta(T)}=
\frac{\beta_0\left(\frac{\lambda}{T}\right)}{\frac{\beta_{\rm
cr}^2\gamma_1}{\sqrt{2T}}(\sqrt{2\lambda}-\gamma)+\mathcal{O}
\left(\frac{\lambda}{T}\right)+o\left(\frac{1}{\sqrt{T}}\right)}=\\
\nonumber =\frac{\sqrt{2T}}{\beta_{\rm
cr}\gamma_1(\sqrt{2\lambda}-\gamma)}&+\mathcal{O}_a(1),
\end{align}
where $\mathcal{O}_a(1)$ denotes a bounded quantity depending on $a$.

Combining \eqref{asymps} and Lemma \ref{opasym} gives
\[
g_T^f\left(\frac{\lambda}{T},y\right)=\sqrt{T}\frac{2\pi\int_{\mathbb{R}^3}
\psi(x)f(x)dx}{\beta_{\rm cr}
\left(\int_{\mathbb{R}^3}v(x)\psi(x)dx\right)^2}\frac{v(y)\psi(y)}{\sqrt{2\lambda}-\gamma}+\tilde{\mathcal{O}}_a(1)f,
\]
as $|\lambda|\leq aT$ and $T\to\infty$, where $\tilde{\mathcal{O}}_a(1)$
is a bounded operator.

For $|\lambda|>aT$, the first term on the right hand side of
\eqref{lemma1statement} is uniformly bounded. By Lemma
\ref{op_bdd} and $\beta(T)\to\beta_{\rm cr}$, so is the left hand
side, and therefore their difference too.
\end{proof}
\noindent{\it Proof of Proposition \ref{cauchy1}}. Fix $\delta>0$
and note that
\begin{align}\label{cont_int}
u_{\beta(T)}(t,x)&=\frac{1}{2\pi
i}\int_{\Gamma(\lambda_0(\beta(T))+\delta/T)} e^{\lambda
t}(R_0(\lambda)g_T^f(\lambda,\cdot))(x)d\lambda=\\ \nonumber
&=\frac{1}{2\pi i}\int_{\Gamma(T\lambda_0(\beta(T))+\delta)}
e^{\lambda\frac{t}{T}}\int_{\mathbb{R}^3}\frac{e^{-\sqrt{2\lambda}
\frac{|x-y|}{\sqrt{T}}}}{2\pi|x-y|}g_T^f\left(\frac{\lambda}{T},y\right)dy\frac{d\lambda}{T}
\end{align}
after a change of variables, where we used the explicit expression
for the kernel of $R_0(\lambda)$, which is just Green's function
for the Laplacian on the whole space. Note that by Lemma
\ref{op_bdd} a), moving the contour is permitted.

By Lemma \ref{princip_eig_inv_asym} and \eqref{betaTasym}, we have
\begin{equation}
\lambda_0(\beta(T))=\frac{1}{\gamma^2\beta_{\rm cr}^4}
(\beta(T)-\beta_{\rm cr})^2+o((\beta(T)-\beta_{\rm cr})^2)=
\frac{\chi^2}{\gamma_1^2\beta_{\rm cr}^4}\frac{1}{T}+o\left(\frac{1}{T}\right),
\end{equation}
and we get
$T\lambda_0(\beta(T))\to\frac{\chi^2}{\gamma_1\beta_{\rm
cr}^4}=\gamma^2/2$. Therefore, for large enough $T$, we can take
the path $\Gamma(\gamma^2/2+\delta)$ as the contour of
integration.

Let us denote the first term in \eqref{lemma1statement} by
$g_{T}^0$, and let $g_{T}^1=g_T-g_T^0$. If $u_{\beta(T)}^0$ stands
for the contour integral \eqref{cont_int} with $g_T^0$ in place of
$g_T$, then we have
\begin{align}\label{main_term_last_step}
u&_{\beta(T)}^0(t,x)= \\
\nonumber&=\frac{1}{\sqrt{T}}\frac{\alpha(f)}{2\pi i}
\int_{\Gamma\left(\frac{\gamma^2}{2}+\delta\right)}
\frac{e^{\lambda\frac{t}{T}}}{\sqrt{2\lambda}-\gamma}
\frac{1}{\int_{\mathbb{R}^3}v(w)\psi(w)dw}\int_{\mathbb{R}^3}
\frac{e^{-\sqrt{2\lambda}\frac{|x-y|}{\sqrt{T}}}}{|x-y|}v(y)\psi(y)dyd\lambda.
\end{align}

It's easy to see by Taylor's formula that for $y$ bounded,  $\epsilon\sqrt{T}\leq|x|\leq\epsilon^{-1}\sqrt{T}$, $\lambda\in\Gamma(\gamma^2/2+\delta)$, there are $C,\alpha>0$ such that for large enough $T$,
\[
\left|\frac{e^{-\sqrt{2\lambda}\frac{|x-y|}{\sqrt{T}}}}{|x-y|}-\frac{e^{-\sqrt{2\lambda}\frac{|x|}{\sqrt{T}}}}{|x|}\right|\leq C\frac{e^{-\alpha\sqrt{\lambda}}}{T}
\]
Plugging this back into \eqref{main_term_last_step}, the first
term gives the main term of \eqref{thm2result}, while the
remainder is easily shown to satisfy \eqref{thm2error} (as $v$ has
compact support, all integrals exist).

The remaining error term is
\[
u_{\beta(T)}^1(t,x)=\frac{1}{2\pi iT}
\int_{\Gamma\left(\frac{\gamma^2}{2}+
\delta\right)}e^{\lambda\frac{t}{T}}
\int_{\mathbb{R}^3}\frac{e^{-\sqrt{2\lambda}\frac{|x-y|}{\sqrt{T}}}}{2\pi|x-y|}(K(\lambda,T)f)(y)dyd\lambda.
\]
Splitting the spatial integration as
\[
u_{\beta(T)}^1(t,x)=\frac{1}{2\pi iT}
\int_{\Gamma(\frac{\gamma^2}{2}+\delta)}
\int_{\{|y|\leq\frac{\epsilon\sqrt{T}}{2}\}}+
\int_{\{|y|>\frac{\epsilon\sqrt{T}}{2}\}}=I_a+I_b,
\]
we get, after making the substitution $\xi:=\lambda t/T$,
\[
I_a=\frac{1}{2\pi i}\frac{1}{T}
\int_{\Gamma(\frac{t}{T}(\gamma^2/2+\delta))}
\int_{|y|\leq\frac{\epsilon\sqrt{T}}{2}}
e^{\xi-\sqrt{2\xi}\frac{|x-y|}{\sqrt{t}}}\frac{1}{2\pi|x-y|}(K(\lambda(\xi),T)f)(y)e^{y^2}e^{-y^2}dyd\xi.
\]
Now $1/|x-y|<2/(\epsilon\sqrt{T})$ yields
\[
|I_a|\leq\frac{||f||_{C_{\rm exp}}C(\epsilon)}{T^{3/2}}.
\]

Before we estimate $I_b$, we need to make the following observation. The contour of integration in all the previous formulas as well as in the expression for $I_b$, can be bent towards the negative real axis. Namely, by $\Gamma'(a)$, we mean a union of two rays emanating from $a$ and that make a $\pm 45$ degree angle with the negative real axis. In all the preceeding formulas, the integration can be performed on either of the corresponding contours $\Gamma$ or $\Gamma'$ since the integrands are analytic and decay along the imaginary axis and decay expononentially in the negative real direction. 

Then the change of variables $\xi:=\lambda t/T$ yields
\[
I_b=\frac{1}{2\pi i}\frac{1}{T}
\int_{\Gamma'(\frac{t}{T}(\gamma^2/2+\delta))}
\int_{|y|>\frac{\epsilon\sqrt{T}}{2}}
e^{\xi-\sqrt{2\xi}\frac{|x-y|}{\sqrt{t}}}\frac{1}{2\pi|x-y|}(K(\lambda(\xi),T)f)(y)e^{y^2}e^{-y^2}dyd\xi.
\]
Setting $x=\sqrt{T}z$ and
$y=\sqrt{T}u$, we get
\[
|I_b|\leq\frac{C_1||f||_{C_{\rm exp}}}{(2\pi)^2}
\frac{e^{-\frac{\epsilon^2}{2}T}}{T^{3/2}}
\int_{\Gamma'(\frac{t}{T}(\gamma^2/2+\delta))}
\int_{|u|>\frac{\epsilon}{2}}\frac{\left|e^{\xi-\sqrt{2\xi}C_2|z-u|}\right|}{|z-u|}e^{-(u^2-\frac{\epsilon^2}{2})T}dud\xi,
\]
from where the exponential decay of $|I_b|$ as $T\to\infty$ and thus the first claim of the Theorem follow.

It is not difficult to deduce \eqref{thm2fundam} from
\eqref{thm2result} after noting that the fundamental solution at
time $t$ is the solution with the initial data
$p_{\beta}(t,y,\delta)$ evaluated at time $t-\delta$. \qed

{\it Proof of Proposition \ref{cauchy2}}. Let
$u_T=p_{\beta(T)}-p_0$. Then
\[
\frac{\partial}{\partial t} u_T=H_{\beta(T)}u_T+\beta v p_0,
\]
 with
initial condition zero. By the Duhamel formula,
\[
u_T(t,y,x)=\int_0^t\int_{\mathbb{R}^3}p_{\beta(T)}(t-s,z,x)\beta
v(z)p_0(s,y,z)dzds,
\]
which can be written as
\begin{align*}
u_T(&t,y,x)=\\
&=\int_0^t\int_{\mathbb{R}^3}\frac{1}{\sqrt{T}}
\frac{\kappa}{2\pi i}\int_{\Gamma(\frac{\gamma^2}{2}+\delta)}
\frac{e^{\lambda\frac{t-s}{T}-\sqrt{2\lambda}
\frac{|x|}{\sqrt{T}}}}{(\sqrt{2\lambda}-\gamma)|x|}
d\lambda\psi(z)\beta(T)v(z)p_0(s,y,0)dzds+h_T(t,y,x)
\end{align*}
where $h_T(t,y,x)$ is an error term.
The first term, denoted by $u^0(t,y,x)$, can be easily seen to
equal
\begin{equation}\label{conv}
u^0_T(t,y,x)=\frac{1}{\sqrt{T}}
\frac{\beta(T)}{\beta_{\rm cr}}\int_0^tw_{T,\lambda,x}(t-s)p_0(s,y,0)ds,
\end{equation}
where
\begin{equation}\label{w}
w_{T,\lambda,x}(t)=\frac{1}{2\pi i}
\int_{\Gamma(\frac{\gamma^2}{2}+\delta)}
\frac{e^{\lambda\frac{t}{T}-\sqrt{2\lambda}
\frac{|x|}{\sqrt{T}}}}{(\sqrt{2\lambda}-\gamma)|x|}d\lambda.
\end{equation}
We can evaluate the convolution in \eqref{conv} in the following
way. First note that \eqref{w} is an inverse transform, while the
transform of $p_0(s,y,0)$ is $e^{-\sqrt{2\lambda}|y|}/2\pi|y|$,
and thus the transform of the convolution is
\[
\frac{T}{2\pi}\frac{e^{-\sqrt{2\lambda}(|x|+|y|)}}{(\sqrt{2\lambda T}-\gamma)|x||y|}.
\]
Applying the inverse formula of the Laplace transform and
substituting $\lambda\to\lambda/T$, we get the main term in
\eqref{thm2asym}.

The remainder term can be written as
\[
h_T(t,y,x)=h^{(1)}_T(t,y,x)+h^{(2)}_T(t,y,x),
\]
where
\begin{align*}
h^{(1)}_T(T,t,y,x)=\int_0^t\int_{\mathbb{R}^3}\frac{1}{\sqrt{T}}
\frac{\kappa}{2\pi i}\int_{\Gamma(\frac{\gamma^2}{2}+\delta)}
\frac{e^{\lambda\frac{t}{T}-\sqrt{2\lambda}
\frac{|x|}{\sqrt{T}}}}{(\sqrt{2\lambda}-\gamma)|x|}d\lambda\cdot\\
\cdot\psi(z)v(z)\beta(T)(p_0(s,y,z)-p_0(s,y,0))d z d s.
\end{align*}

Using the same Laplace transform trick, this can be shown to equal
\[
\frac{1}{\sqrt{T}}\frac{\beta(T)}{\beta_{\rm cr}}\frac{1}{2\pi i}
\int_{\Gamma(\frac{\gamma^2}{2}+\delta)}
\frac{e^{\lambda\frac{t}{T}}}{\sqrt{2\lambda}-\gamma}
\frac{e^{-\sqrt{2\lambda}\frac{|x|}{\sqrt{T}}}}{|x|}
\int_{\mathbb{R}^3}\left(\frac{e^{-\sqrt{2\lambda}
\frac{|y-z|}{\sqrt{T}}}}{|y-z|}-\frac{e^{-\sqrt{2\lambda}
\frac{|y|}{\sqrt{T}}}}{|y|}\right)dzd\lambda,
\]
and this can be shown to satisfy \eqref{err_term} the same way the
main term in \eqref{thm2result} followed from
\eqref{main_term_last_step}. The last remaining term $h_T^{(2)}$
containing the error term of \eqref{thm2fundam} can easily be
shown to satisfy \eqref{err_term}. \qed

\noindent{\it Proof of Proposition \ref{partition_sum}}. Applying
the Laplace transform techniques, it is not difficult to show (see
Lemma 7.1 in \cite{CKMV09}) that
\[
Z_{\beta(T),t}(x)-1=-\frac{1}{2\pi
i}\int_{\Gamma\left(\lambda_0(\beta(T))+\frac{\delta}{T}\right)}
\frac{e^{\lambda t}}{\lambda}(R_{\beta(T)}(\lambda)(\beta
v))(x)d\lambda.
\]
Using this formula, one can prove the claim following the same
steps as in the above proof of Proposition \ref{cauchy1}. \qed

The next lemma easily follows from the Feynman-Kac formula.
\begin{lemma}\label{fin_dim_distr}
For $0=t_0<t_1<...<t_n\leq T$,
\[
\mathrm{P}_{\beta,T}(x(t_1)\in dx_1,...,x(t_n)\in
dx_n)=\frac{\prod_{i=0}^{n-1}p_{\beta}
(t_{i+1}-t_i,x_i,x_{i+1})Z_{\beta,T-t_n}(x_n)}{Z_{\beta,T}}dx_1...dx_n.
\]
\end{lemma}

\noindent {\it Proof of Theorem \ref{main_result}}. By Lemma
\ref{fin_dim_distr}, the finite-dimensional densities of the
measure $f_{T^{-1}}\mathrm{P}_{\beta(T),T}$ for $0<t_1<...<t_n\leq
1$ and $x_1,...,x_n\in\mathbb{R}^3$ are
\begin{align*}
&\rho_{t_1,...,t_n}^T(x_1,...,x_n)=\\
&=T^{3n/2}p_{\beta(T)}(t_1T,0,x_1T^{1/2})...p_{\beta(T)}
((t_n-t_{n-1})T,x_{n-1}T^{1/2},x_nT^{1/2})
\frac{Z_{\beta(T),T(1-t_n)}(x_nT^{1/2})}{Z_{\beta(T),T}(0)}.
\end{align*}
Let's introduce
\[
p^T(s,t,y,x)=p_{\beta(T)}(T(t-s),yT^{1/2},xT^{1/2})
\]
and
\begin{align*}
R^T(s,t,y,x)=T^{3/2}\cdot\left\{\begin{array}{cc}
p_{\beta(T)}^T(s,t,y,x)\frac{Z_{\beta(T),T(1-t)}(xT^{1/2})}{Z_{\beta(T),T(1-s)}(yT^{1/2})}&t<1,\\
\frac{p_{\beta(T)}^T(s,1,y,x)}{Z_{\beta(T),T(1-s)}(yT^{1/2})}&t=1.
\end{array}\right.
\end{align*}

Then it is not hard to show that
\[
\rho_{t_1,...,t_n}^T(x_1,...,x_n)=R^T(0,t_1,0,x_1)\cdot ...\cdot R^T(t_{n-1},t_n,x_{n-1},x_n).
\]
Note that by Proposition \ref{cauchy1} and Proposition \ref{cauchy2}, we have for $x\neq 0$
\begin{equation}\label{limit1}
\lim_{T\to\infty}Tp_{\beta(T)}(Tt,0,xT^{1/2})=\frac{\kappa\psi(0)}{2\pi
i} \int_{\Gamma(\frac{\gamma^2}{2}+\delta)}\frac{e^{\lambda
t-\sqrt{2\lambda}|x|}}{(\sqrt{2\lambda}-\gamma)|x|}d\lambda,
\end{equation}
while for $x,y\neq 0$,
\begin{align}\label{limit2}
\lim_{T\to\infty}T^{3/2}p_{\beta(T)}(T&t,yT^{1/2},xT^{1/2})=\\
\nonumber&=\frac{e^{-\frac{|x-y|^2}{2t}}}{(2\pi
t)^{3/2}}+\frac{1}{4\pi^2 i}
\int_{\Gamma(\frac{\gamma^2}{2}+\delta)}\frac{e^{\lambda
t-\sqrt{2\lambda}(|x|+|y|)}}{(\sqrt{2\lambda}-\gamma)|x||y|}d\lambda.
\end{align}

By Proposition \ref{partition_sum}, for $x\neq 0$,
\begin{equation}\label{limit3}
Z_{\beta(T),Tt}(xT^{1/2})\to \overline{Z}_{\gamma,t}(x),\qquad
T\to\infty.
\end{equation} 
Using this, \eqref{limit2} and \eqref{fund_zerorange}, it follows that for $x,y\neq 0$,
\[
\lim_{T\to\infty}R^T(s,t,y,x)=R(s,t,y,x):=\overline{p}_{\gamma}(t-s,y,x)\frac{\int_{\mathbb{R}^3}\overline{p}_{\gamma}(t,1,x,z)dz}{\int_{\mathbb{R}^3}\overline{p}_{\gamma}(s,1,y,z)dz}
\]

Moreover, as follows from \eqref{limit1} and \eqref{fund_zerorange}, it is not difficult to see that if $x\neq 0$,
\[
\lim_{T\to\infty}R^T(s,t,0,x)=R(s,t,0,x):=
\lim_{y\to 0}\overline{p}_{\gamma}(t-s,y,x)\frac{\int_{\mathbb{R}^3}\overline{p}_{\gamma}
(t,1,x,z)dz}{\int_{\mathbb{R}^3}\overline{p}_{\gamma}(s,1,y,z)dz}.
\]

Since $R(s,t,y,x)$ is the transition density of the polymer under the measure $Q_{\gamma}$ (as discussed in Section 1 and as shown in \cite{CKMV10}), this implies the convergence of the finite-dimensional distributions and the result will follow once tightness is shown. On
the other hand, the proof of tightness is only a slight
modification of the proof of Lemma 10.5  in \cite{CKMV09} (where the case of fixed $\beta$ was treated), so we
don't provide it here. \qed

\end{document}